\documentclass[onecolumn]{article}
\usepackage{amssymb,amsmath,amsthm}
\usepackage{microtype}
\usepackage{enumitem}
\usepackage{xcolor}
\usepackage{graphicx}
\usepackage{bm}
\usepackage{bbm}
\usepackage{algorithm}
\usepackage[noend]{algpseudocode}
\algnewcommand\algorithmicinit{\textbf{Initialization:}}
\algnewcommand\init{\item[\algorithmicinit]}
\algnewcommand\algorithmicawake{\textsf{\textit{{AWAKE}}}}
\algnewcommand\awake{\item[\algorithmicawake]}
\algnewcommand\algorithmicidle{\textsf{\textit{{IDLE}}}}
\algnewcommand\idle{\item[\algorithmicidle]}

\algnewcommand\algorithmicinput{\textbf{Input:}}
\algnewcommand\inpt{\item[\algorithmicinput]}

\algnewcommand\algorithmicoutput{\textbf{Output:}}
\algnewcommand\outpt{\item[\algorithmicoutput]}

\renewcommand{\natural}{{\mathbb{N}}}
\newcommand{\real}{{\mathbb{R}}}

\newcommand{\until}[1]{\{1,\ldots,#1\}} 

\newcommand{\BB}{\mathcal{B}}

\newcommand{\EE}{\mathcal{E}}
\newcommand{\GG}{\mathcal{G}}

\newcommand{\VV}{\mathcal{V}}
\newcommand{\WW}{\mathcal{W}}

\newcommand{\m}{\mathop{\textrm{minimize}}}

\newcommand{\R}{\mathbb{R}}

\newcommand{\NNii}{\mathcal{N}_{i}^{\text{in}}}

\newcommand{\NNio}{\mathcal{N}_{i}^{\text{out}}}
\newcommand{\lxi}{x_j^k{\mid}_i}
\newcommand{\E}{\mathbb{E}}

\newcommand{\Pb}{\text{P}}

\makeatletter
\newcommand{\StatexIndent}[1][3]{%
  \setlength\@tempdima{\algorithmicindent}%
  \Statex\hskip\dimexpr#1\@tempdima\relax}
\makeatother

\makeatletter
\newcommand{\pushright}[1]{\ifmeasuring@#1\else\omit\hfill$\displaystyle#1$\fi\ignorespaces}
\newcommand{\pushleft}[1]{\ifmeasuring@#1\else\omit$\displaystyle#1$\hfill\fi\ignorespaces}
\makeatother

\newtheorem{theorem}{Theorem}

\newtheorem{assumption}{Assumption}

\newtheorem{remark}{Remark}[section]
  {\list{}{\leftmargin=0.3in}\item[]\noindent\rule[0.5ex]{\linewidth}{0.5pt}}%
  {\noindent\rule[0.5ex]{\linewidth}{0.5pt}\endlist\clearpage}

\newcommand\oprocendsymbol{\hbox{$\square$}}
\newcommand\oprocend{\relax\ifmmode\else\unskip\hfill\fi\oprocendsymbol}

\title{Distributed Submodular Minimization via Block-Wise Updates and Communications\footnote{This result is part of
  a project that has received funding from the European Research Council (ERC) under
  the European Union's Horizon 2020 research and innovation programme
  (grant agreement No 638992 - OPT4SMART).}}

\author{Andrea Testa, Francesco Farina, Giuseppe Notarstefano}

\date{\small Department of Electrical, Electronic and Information Engineering,\\
Alma Mater Studiorum Universit\`{a} di Bologna, Bologna, Italy\\
  $\{$\texttt{a.testa}, \texttt{franc.farina},
  \texttt{giuseppe.notarstefano}$\}$\texttt{@unibo.it}
}

\begin{document}
\maketitle
\begin{abstract}
In this paper we deal with a network of computing agents with local processing and neighboring communication capabilities that aim at solving (without any central unit) a submodular optimization problem. The cost function is the sum of many local submodular functions and each agent in the network has access to one function in the sum only. In this \emph{distributed} set-up, in order to preserve their own privacy, agents communicate with neighbors but do not share their local cost functions. We propose a distributed algorithm in which agents resort to the Lov\`{a}sz extension of their local submodular functions and perform local updates and communications in terms of single blocks of the entire optimization variable. Updates are performed by means of a greedy algorithm which is run only until the selected block is computed, thus resulting in a reduced computational burden.
The proposed algorithm is shown to converge in expected value to the optimal cost of the problem, and an approximate solution to the submodular problem is retrieved by a thresholding operation. As an application, we consider a distributed image segmentation problem in which each agent has access only to a portion of the entire image. While agents cannot segment the entire image on their own, they correctly complete the task by cooperating through the proposed distributed algorithm.
\end{abstract}

\section{Introduction}
Many combinatorial problems in machine learning can be cast as the minimization of submodular functions (i.e., set functions that exhibit a diminishing marginal returns property). Applications include isotonic regression, image segmentation and reconstruction, and semi-supervised clustering (see, e.g.,~\cite{bach2013learning}).

In this paper we consider the problem of minimizing in a distributed fashion (without any central unit) the sum of $N\in\natural$ submodular functions, i.e.,
\begin{equation}\label{pb:standard}
  \begin{aligned}
      & \m_{X\subseteq V}
      & & F(X)=\sum_{i=1}^N F_i(X)
  \end{aligned}
\end{equation}
where $V=\until{n}$ is called the \emph{ground set} and the functions $F_i$ are submodular.

We consider a scenario in which problem~\eqref{pb:standard} is to be solved by $N$ peer agents communicating locally and performing local computations.
The communication is modeled as a \emph{directed} graph $\GG=(\VV,\EE)$, where $\VV=\until{N}$ is the set of agents and $\EE\subseteq\VV\times\VV$ is the set of directed edges in the graph. Each agent $i$ \emph{receives} information only from its in-neighbors, i.e., agents $j\in
\NNii\triangleq\{j\mid (j,i)\in\EE\}\cup \{i\}$, while it \emph{sends} messages only to its out-neighbors $j\in\NNio\triangleq\{j\mid (i,j)\in\EE\}\cup\{i\}$, where we have included agent $i$ itself in these sets.
In this set-up, each agent knows only a portion of the entire optimization problem. Namely, agent $i$, knows the function $F_i(X)$ and the set $V$ only. Moreover, the local functions $F_i$ must be maintained private by each agent and cannot be shared.

In order to give an insight on how the proposed scenario arises, let us introduce the distributed image segmentation problem that we will consider later on as a numerical example. Given a certain image to segment, the ground set $V$ consists of the pixels of such an image. We consider a scenario in which each of the $N$ agents in the network has access to only a portion $V_i\subseteq V$ of the image. In Figure~\ref{fig:cams} a concept with the associated communication graph is shown. Given $V_i$, the local submodular functions $F_i$ are constructed by using some locally retrieved information, like pixel intensities. While agents do not want to share any information on how they compute local pixel intensities (due to, e.g., local proprietary algorithms), their common goal is to correctly segment the entire image.

\begin{figure}[t]
  \centering
  \includegraphics[width=.38\textwidth]{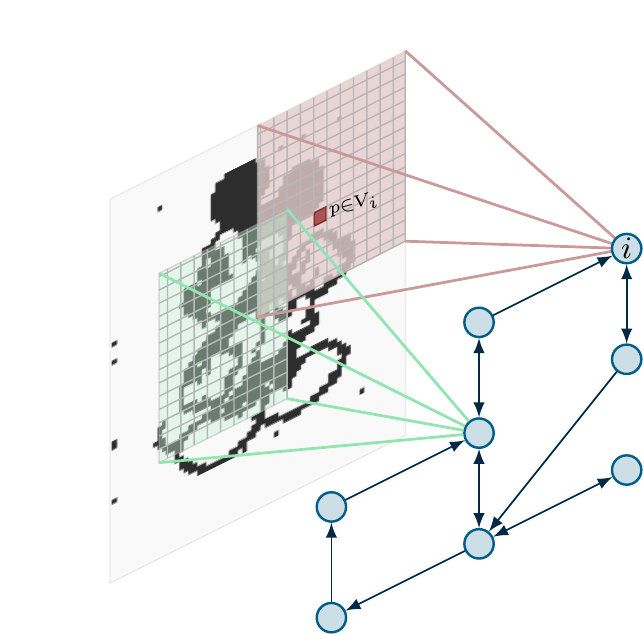}
  \caption{Distributed image segmentation set-up.
  Network agents (blue nodes) have access only to a subset (colored grids)
  of the whole image pixels. Directed arcs between nodes represent
  the communication links.
  }
  \label{fig:cams}
\end{figure}

Such a distributed set-up is motivated by the modern organization of data and computational power. It is extremely common for computational units to be connected in networks, sharing some resources, while keeping other private, see, e.g.,~\cite{stone2000multiagent,decker1987distributed}. Thus, distributed algorithms in which agents do not need to disclose their own private data will represent a novel disruptive technology.
This paradigm has received significant attention in the last decade in the area of control and signal processing,~\cite{ahmed2016distributed,chen2018internet}.

\paragraph*{Related work}
Submodular minimization problems can be mainly addressed in two ways. On the one hand, a number of combinatorial algorithms have been proposed~\cite{iwata2001combinatorial,iwata2009simple}, some based on graph-cut algorithms~\cite{jegelka2011fast} or relying on problems with a particular structure~\cite{kolmogorov2012minimizing}. On the other hand, convex optimization techniques can be exploited to face submodular minimization problems by resorting the so called Lov\`{a}sz extension. Many specialized algorithms have been developed in the last years by building on the particular properties of submodular functions (see~\cite{bach2013learning} and reference therein).
In this paper we focus on the problem of minimizing the sum of many submodular functions, which has received attention in many works~\cite{stobbe2010efficient,kolmogorov2012minimizing,jegelka2013reflection,fix2013structured,nishihara2014convergence}.
In particular, centralized algorithms have been proposed based on smoothed convex minimization~\cite{stobbe2010efficient} or alternating projections and splitting methods~\cite{jegelka2013reflection}, whose convergence rate is studied in~\cite{nishihara2014convergence}. This problem structure typically arises, for example, in Markov Random Fields (MRF) Maximum a-Posteriori (MAP) problems~\cite{shanu2016min,fix2013structured}, a notable example of which is image segmentation.

While a vast literature on distributed continuous optimization has been developed in the last years (see, e.g.,~\cite{notarstefano2019distributed}), distributed approaches for tackling (submodular) combinatorial optimization problems started to appear only recently.
Submodular maximization problems have been treated and approximately solved in a distributed way in several works
~\cite{kim2011distributed,mirzasoleiman2013distributed,bogunovic2017distributed,williams2017decentralized,gharesifard2017distributed,grimsman2017impact}. In particular, distributed submodular maximization subject to matroid constraints is addressed in~\cite{williams2017decentralized,gharesifard2017distributed}, while in~\cite{grimsman2017impact}, the authors handle the design of communication  structures maximizing the worst case efficiency of the well-known greedy algorithm for submodular maximization when applied over networks.
Regarding distributed algorithms for submodular minimization problems, they have not received much attention yet.
In~\cite{jaleel2018real} a distributed subgradient method is proposed, while in~\cite{testa2018distributed} a greedy column generation algorithm is given. All these approaches involve the communication/update of the entire decision variable at each time instant. This can be an issue when the decision variable is extremely large. Thus, block-wise approaches like those proposed in~\cite{notarnicola2018distributed} should be explored.

\paragraph*{Contribution and organization}
The main contribution of this paper is the MIxing bloCKs and grEedY (MICKY) method, i.e., a distributed block-wise algorithm for solving problem~\eqref{pb:standard}. At any iteration, each agent computes a weighted average on local copies of neighbors solution estimates. Then, it selects a random block and performs an ad-hoc (block-wise) greedy algorithm (based on the one in~\cite[Section~3.2]{bach2013learning}) until the selected block is updated. Finally, based on the output of the greedy algorithm, the selected block of the local solution estimate is updated and broadcast to the out-neighbors.
The proposed algorithm is shown to produce cost-optimal solutions in expected value by showing that it is an instance of the Distributed Block Proximal Method presented in~\cite{farina2019arXivProximal}. In fact, the partial greedy algorithm performed on the local submodular cost function $F_i$ is shown to compute a block of a subgradient of its Lov\`{a}sz extension.

A key property of this algorithm is that each agent is required to update and transmit only one block of its solution estimate.
In fact, it is quite common for networks to have communication bandwidth restrictions. In these cases the entire state variable may not fit the communication channels and, thus, standard distributed optimization algorithms cannot be applied. Furthermore, the greedy algorithm can be very time consuming when an oracle for evaluating the submodular functions is not available and, hence, halting it earlier can reduce the computational load.

The paper is organized as follows. The distributed algorithm is presented and analyzed in Section~\ref{sec:algo}, and it is tested on a distributed image segmentation problem in Section~\ref{sec:numerical}.

\vspace{-2ex}

\paragraph*{Notation and definitions}
Given a vector $x\in\R^n$, we denote by $x_\ell$ the $\ell$-th entry of $x$.
Let $V$ be a finite, non-empty set with cardinality $|V|$. We denote by $2^V$ the set of
all its $2^{|V|}$ subsets.
Given a set $X\subseteq V$, we denote by $\mathbf{1}_X\in\real^{|V|}$
its indicator vector, defined as
$\mathbf{1}_{X_\ell}=1$ if $\ell \in X$, and $0$ if $\ell \not\in X$.
A set function $F:2^V\to \R$ is said to be submodular if it exhibits the diminishing marginal returns property, i.e., for all $A,B\subseteq V$, $A\subseteq B$ and for all $j\in V\setminus B$, it holds that $F(A\cup\{j\})- F(A)\geq F(B\cup\{j\})- F(B)$. In the following we assume $F(X)<\infty$ for all $X\subseteq V$ and, without loss of generality, $F(\emptyset)=0$.
Given a submodular function $F:2^V\to \R$, we define the associated \emph{base polyhedron} as $\BB(F):=\{w\in\real^n\mid \sum_{\ell\in X} w_\ell \leq F(X)\; \forall X\in 2^V,\; \sum_{\ell\in V} w_\ell = F(V)\}$ and by $f(x)=\max_{w\in\BB(F)}w^\top x$ the Lov\`{a}sz extension of $F$.

\section{Distributed algorithm}\label{sec:algo}
\subsection{Algorithm description}
In order to describe the proposed algorithm, let us introduce the following nonsmooth convex optimization problem
\begin{equation}\label{pb:lovasz}
    \begin{aligned}
        & \m_{x\in [0,1]^n}
        & & f(x)=\sum_{i=1}^N f_i(x)
    \end{aligned}
\end{equation}
where $f_i(x):\R^n\to\R$ is the Lov\`{a}sz extension of $F_i$ for all $i\in\until{N}$. It can be shown that solving problem~\eqref{pb:lovasz} is equivalent to solving problem~\eqref{pb:standard} (see, e.g.,~\cite{lovasz1983submodular} and~\cite[Proposition~3.7]{bach2013learning}).
In fact, given a solution $x^\star$ to problem~\eqref{pb:lovasz}, a solution $X^\star$ to problem~\eqref{pb:standard} can be retrieved by thresholding the components of $x^\star$ at an arbitrary $\tau\in[0,1]$ (see~\cite[Theorem 4]{bach2019submodular}), i.e.,
\begin{align}
\label{eq:thresh}
X^\star = \{\ell\mid x^\star_\ell>\tau\}.
\end{align}
Notice that, given $F_i$ in problem~\eqref{pb:standard}, each agent $i$ in the network is able to compute $f_i$, thus, in the considered distributed set-up, problem~\eqref{pb:lovasz} can be addressed in place of problem~\eqref{pb:standard}. Moreover, since $F_i$ is submodular for all $i$, then $f_i$ is a continuous, piece-wise affine, nonsmooth convex function, see, e.g.,~\cite{bach2013learning}.

In order to compute a single block of a subgradient of $f_i$, each agent $i$ is equipped with a local routine (reported next), that we call \textsc{BlockGreedy} and that resembles a local (block-wise) version of the greedy algorithm in~\cite[Section~3.2]{bach2013learning}.
This routine takes as inputs a vector $y$ and the required block $\ell$, and returns the $\ell$-th block of a subgradient $g_i$ of $f_i$ at $y$. For the sake of simplicity, suppose $l$ is a single component block. Moreover, assume to have a routine \textsc{PartialSort} that generates an ordering $\{m_1,\ldots,m_p\}$ such that
$y_{m_1}\geq\ldots\geq y_{m_p}$,
$m_p=\ell$ and
$y_{r}\leq y_\ell$ for each
$r\in\until{n}\setminus\{m_1,\ldots,m_p\}$.
Then, the \textsc{BlockGreedy} algorithm reads as follows.
\begin{algorithm}
  \renewcommand{\thealgorithm}{}
  \floatname{algorithm}{Routine}
    \begin{algorithmic}[h!]
    \small
      \inpt $y$, $\ell$

      \State Obtain a partial order via
      $$
        \{m_1,\ldots,m_{p-1}, m_{p}=\ell\}= \textsc{PartialSort} (y)
      $$

      \State Evaluate $g_{i,m_p}$ as %
  \begin{align*}
    g_{i,\ell}
    \!\!=\!
    \begin{cases}
    F_i(\{\ell\}),
    & \text{if } p = 1
    \\
    F_i(\{m_1 \ldots m_{p-1}, \ell \}) \!-\! F_i(\{m_1 \ldots m_{p-1} \}), \!\!
    & \text{otherwise}%
    \end{cases}
  \end{align*}
    \outpt $g_{i,\ell}$
    \end{algorithmic}
  \caption{\textsc{BlockGreedy}$(y,\ell)$ for agent $i$}
  \label{alg:local_greedy}
\end{algorithm}

The MICKY algorithm works as follows. Each agent stores a local solution estimate $x_i^k$ of problem~\eqref{pb:lovasz} and, for each in-neighbor $j\in\NNii$, a local copy of the corresponding solution estimate $\lxi$. At the beginning, each node selects the initial condition $x_i^0$ at random in $[0,1]^n$ and shares it with its out-neighbors. We associate to the communication graph $\GG$ a weighted adjacency matrix $\WW\in\R^{N\times N}$ and we denote with $w_{ij}=[\WW]_{ij}$ the weight associated to the edge $(j,i)$.
At each iteration $k$,
agent $i$ performs three tasks:
\begin{enumerate}[label=(\roman*)]
  \item it computes a weighted average $y_i^k = \sum_{j\in\NNii} w_{ij} \lxi$;
  \item it picks randomly (with arbitrary probabilities bounded away from $0$) a block $\ell_i^k\in\until{n}$ and performs the \textsc{BlockGreedy}$(y_i^k,\ell_i^k)$;
\item it updates $x_{i,\ell_i^k}^{k+1}$ according to~\eqref{eq:up_alg}, where $\Pi_{[0,1]}[\cdot]$ is the projector on the set $[0,1]$ and $\alpha_i^k\in(0,1)$, and broadcasts it to its out-neighbors $j\in\NNio$.
\end{enumerate}
Agents halt the algorithm after $K>0$ iterations and recover the local estimates $X_i^{\text{end}}$ of the set solution to problem~\eqref{pb:standard} by thresholding the value of $x_i^K$ as in~\eqref{eq:thresh}. Notice that, in order to avoid to introduce additional notation, we have assumed each block of the optimization variable to be scalar (so that blocks are selected in $\until{n}$). However, blocks of arbitrary sizes can be used (as shown in the subsequent analysis).
A pseudocode of the proposed algorithm is reported in the next table.

\begin{algorithm}[H]
  \renewcommand{\thealgorithm}{}
  \floatname{algorithm}{Algorithm}
	\begin{algorithmic}[H!]
	\small
          \init $x_i^0$
		\item[]
    \For{$k=1,\dots,K-1$}
      \State \textsc{Update} for all $j\in\NNii$
        \begin{equation}\label{eq:xl_update}
            x_{j,\ell}^k{\mid}_i =
            \begin{cases}
                x_{j,\ell}^k, &\text{if }\ell=\ell_j^{k-1}\\
                x_{j,\ell}^{k-1}{\mid}_i, &\text{otherwise}
            \end{cases}
        \end{equation}
      \State \textsc{Compute}
      \begin{equation}
        y_i^k = \sum_{j\in\NNii} w_{ij} \lxi
      \end{equation}
      \State \textsc{Pick} randomly a block $\ell_i^{k}\in\until{n}$ %
      \State \textsc{Compute}
      \begin{equation}
        g_{i,\ell_i^k}^k = \textsc{BlockGreedy}(y_i^k,\ell_i^k)
      \end{equation}
      \State \textsc{Update}
      \begin{equation}\label{eq:up_alg}
        x_{i,\ell}^{k+1} = \begin{cases}
          \Pi_{[0,1]}\left[x_{i,\ell_i^k}^{k} - \alpha_i^k g_{i,\ell_i^k}^k\right]&\text{if }\ell=\ell_i^k\\
          x_{i,\ell}^{k}&\text{otherwise}
          \end{cases}
      \end{equation}
      \State \textsc{Broadcast} $x_{i,\ell_i^k}^{k+1}$ to all $j\in\NNio$
    \EndFor
    \textsc{Thresholding}
    \begin{equation}\label{eq:reconstruct}
      X_i^{\text{end}} = \{\ell\mid x_{i,\ell}^K > \tau\}
    \end{equation}

	\end{algorithmic}
	\caption{MICKY\! (Mixing\! Blocks\! and\! Greedy\! Method)}\label{alg:DSM}
\end{algorithm}

\subsection{Discussion}
The proposed algorithm possesses many interesting features. Its distributed nature requires agents to communicate only with their direct neighbors, without resorting to multi-hop communications. Moreover, all the local computations involve locally defined quantities only. In fact, stepsize sequences and block drawing probabilities are locally defined at each node.

Regarding the block-wise updates and communications, they bring benefits in two areas.
Communicating single blocks of the optimization variable, instead of the entire one, can significantly reduce the communication bandwidth required by each agent in broadcasting their local estimates. This makes the proposed algorithm implementable in networks with communication bandwidth restrictions. Moreover, the classical greedy algorithm requires to evaluate $|V|$ times the submodular function in order to produce a subgradient. When $|V|$ is very high and an oracle for evaluating functions $F_i$ is not available, this can be a very time consuming task.
For example, in the example application in Section~\ref{sec:numerical}, we will resort to the minimum graph cut problem. Evaluating the value of a cut for a graph in which $E\subseteq V\times V$ is the set of arcs, requires a running-time $O(|E|)$.
In the \textsc{BlockGreedy} routine, in contrast with what happens in the standard greedy routine,
the sorting operation is (possibly) performed only on a part of the entire vector $y$, i.e., until the $\ell$-th component has been sorted. Thus, our routine evaluates the $\ell$-th component of the subgradient in at most two evaluations of the submodular function.

\subsection{Analysis}
In order to state the convergence properties of the proposed algorithm,
let us make the following two assumptions on the communication graph and the associated weighted adjacency matrix $\WW$.
\begin{assumption}[Strongly connected graph]\label{assumption:graph}
  The digraph $\GG=(\VV,\EE,\WW)$ is strongly connected.\oprocend
\end{assumption}
\begin{assumption}[Doubly stochastic weight matrix]\label{assumption:stochastic}
    For all $i,j\in\VV$, the weights $w_{ij}$ of the weight matrix $\WW$ satisfy
    \begin{enumerate}[label=(\roman*)]
        \item if $i\neq j$, $w_{ij}>0$ if and only if $j\in\NNii$;
        \item there exists a constant $\eta>0$ such that $w_{ii}\geq\eta$ and if $w_{ij}>0$, then $w_{ij}\geq\eta$;
        \item $\sum_{j=1}^N w_{ij}=1$ and $\sum_{i=1}^N w_{ij}=1$.\oprocend
    \end{enumerate}
\end{assumption}
The above two assumptions are very common when designing distributed optimization algorithms. In particular, Assumption~\ref{assumption:graph} guarantees that the information is spread through the entire network, while Assumption~\ref{assumption:stochastic} assures that each agent gives sufficient weight to the information coming from its in-neighbors.

Let $\bar{x}^k\triangleq\frac{1}{N}\sum_{i=1}^N x_i^k$ be the average over the agents of the local solution estimates at iteration $k$ and define $f_{best}(x_i^k)\triangleq\min_{r\leq k} \E[f(x_i^r)]$.
Then, in the next result, we show that by cooperating through the proposed algorithm all the agents agree on a common solution and the produced sequences $\{x_i^t\}$ are asymptotically cost optimal in expected value when $K\to\infty$.
\begin{theorem}\label{thm:conv}
  Let Assumptions~\ref{assumption:graph} and~\ref{assumption:stochastic} hold and let $\{x_i^k\}_{k\geq 0}$ be the sequences generated through the MICKY algorithm. Then,
    if the sequences $\{\alpha_i^k\}$ satisfy
	\begin{equation}
		\sum_{k=0}^\infty \alpha_i^k = \infty, \qquad \sum_{k=0}^\infty (\alpha_i^k)^2 < \infty, \qquad \alpha_i^{k+1}\leq\alpha_i^k\label{eq:alpha}
	\end{equation}
  for all $k$ and all $i\in\VV$, it holds that,
  \begin{equation}\label{eq:consensus}
		\lim_{k\to\infty}\E[\|x_i^k-\bar{x}^k\|]=0,
  \end{equation}
  and
	\begin{equation}\label{eq:optimality}
		\lim_{k\to\infty}f_{best}(x_i^k)=f(x^\star),
	\end{equation}
	being $x^\star$ the optimal solution to~\eqref{pb:lovasz}.
\end{theorem}
\begin{proof}
  By using the same arguments used in~\cite[Lemma~3.1]{farina2019arXivProximal}, it can be shown that $\lxi=x_j^k$ for all $k$ and all $i,j\in\VV$. Then~\eqref{eq:consensus} follows from~\cite[Lemma~5.11]{farina2019arXivProximal}.
  Moreover, as anticipated, it can be shown that $g_{i,\ell_i^k}^k$ is the $\ell_i^k$-th block of a subgradient of the function $f_i(x)$ in problem~\eqref{pb:lovasz} (see, e.g.,~\cite[Section~3.2]{bach2013learning}).
  In fact, being $f_i$ defined as the support function of the base polyhedron $\BB(F_i)$, i.e., $f_i(x)=\max_{w\in\BB(F_i)}w^\top x$, the greedy algorithm~\cite[Section~3.2]{bach2013learning} iteratively computes a subgradient of $f_i$ component by component. Moreover, subgradients of $f_i$ are bounded by some constant $G<\infty$, since every component of a subgradient of $f_i$ is computed as the difference of $F_i$ over two different subsets of $V$.
  Given that, the proposed algorithm can be seen as a special instance of the Distributed Block Proximal Method in~\cite{farina2019arXivProximal}. Thus, since Assumptions~\ref{assumption:graph} and~\ref{assumption:stochastic} holds, it inherits all the convergence properties of the Distributed Block Proximal Method and under the assumption of diminishing stepsizes~\eqref{eq:alpha} respectively, the result in~\eqref{eq:optimality} follows (see~\cite[Theorem~5.15]{farina2019arXivProximal}). \oprocend
\end{proof}

Notice that the result in Theorem~\ref{thm:conv} does not say anything about the convergence of the sequences $\{x_i^k\}$, but only states that if diminishing stepsizes are employed, asymptotically these sequences are consensual and cost optimal in expected value.
Despite that, from a practical point of view, two facts typically happen. First, agents approach consensus, i.e., for all $i\in\until{N}$, the value $\|x_i^k-\bar{x}^k\|$ becomes small, extremely fast, so that they all agree on a common solution. Second, if the number of iterations $K$ in the algorithm is sufficiently large, the value of $x_i^K$ is a good solution to problem~\eqref{pb:lovasz}.
Then, given $x_i^K$, each agent can reconstruct a set solution to problem~\eqref{pb:standard} by using~\eqref{eq:reconstruct} and, in order to obtain the same solution for all the agents, we consider a unique threshold value, known to all the agents, $\tau\in[0,1]$.

\begin{remark}
	Notice that, by resorting to classical arguments, it can be easily shown from the analysis in~\cite{farina2019arXivProximal} that the convergence rate of $f_{best}$ in Theorem~\ref{thm:conv} is sublinear (with explicit rate depending on the actual stepsize sequence). Moreover, if constant stepsizes are employed, convergence of $f_{best}$ to the optimal solution is attained in expected value with a constant error with rate $O(1/k)$~\cite[Theorem~2]{farina2019arXivProximal}.
\end{remark}

\section{Cooperative image segmentation}\label{sec:numerical}

Submodular minimization has been widely applied to
computer vision problems as image classification, segmentation
and reconstruction, see,
e.g.,~\cite{stobbe2010efficient,jegelka2013reflection,
greig1989exact}.
In this section, we consider a binary image segmentation problem in which $N=8$ agents have to cooperate in order to separate an object from the background in an image of size $D\times D$ pixels (with $D=64$). Each agent has access only to a portion of the entire image, see Figure~\ref{fig:agent_start}, and can communicate according to the graph reported in the figure.

Before giving the details of the distributed experimental set-up let us introduce how such a problem is usually treated in a centralized way, i.e., by casting it into a $s$--$t$ minimum cut problem.

\subsection{$s$--$t$ minimum cut problem}
Assume the entire $D\times D$ image be available for segmentation, and denote as $V=\until{D^2}$ the set of pixels.
As shown, e.g., in~\cite{greig1989exact,boykov2006graph} this problem can be reduced
to an equivalent $s$--$t$ minimum cut problem, which can be approached by submodular minimization techniques.

\begin{figure}[]
  \centering
  \includegraphics[width=0.85\columnwidth]{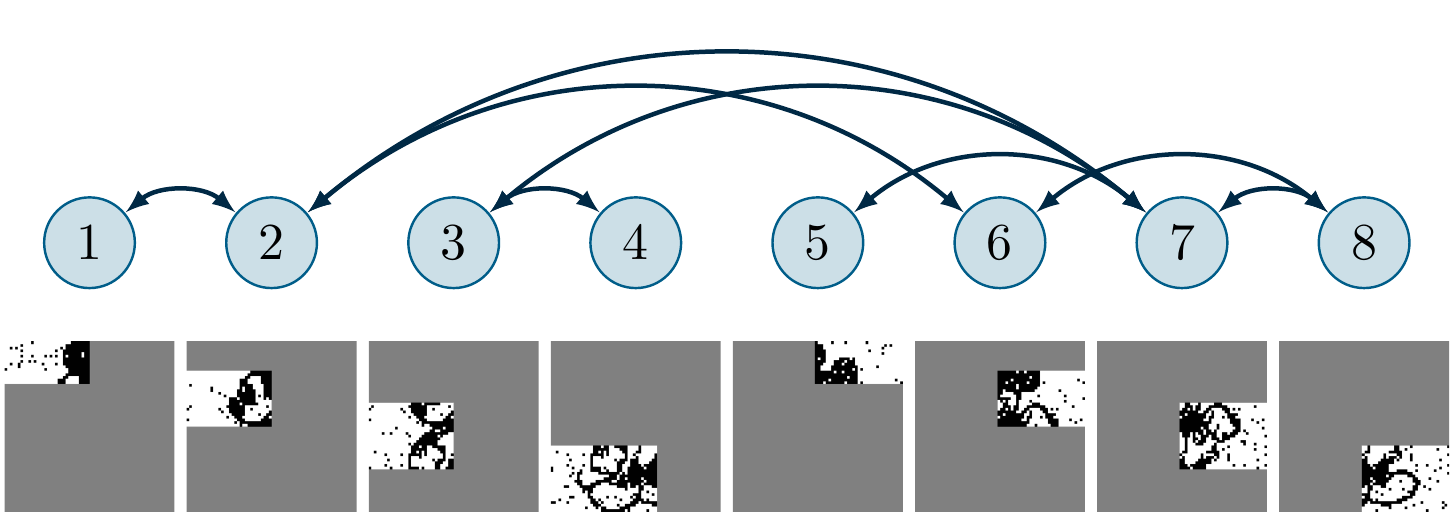}    %
  \caption{Cooperative image segmentation. The considered communication graph is depicted on top, where agents are represented by blue nodes. Under each node, the portion of the image accessible by the corresponding agent is depicted.}
  \label{fig:agent_start}
\end{figure}

More in detail, this approach is based on the construction of a weighted digraph $G_{s-t} = (V_{s-t},E_{s-t}, A_{s-t})$, where $V_{s-t}=\{1,\ldots,D^2,s,t\}$ is the set of nodes, $E_{s-t}\subseteq V_{s-t}\times V_{s-t}$ is the edge set and $A_{s-t}$ is a positive weighted adjacency matrix.
There are two sets of directed edges $(s,p)$ and $(p,t)$, with positive weights $a_{s,p}$ and $a_{p,t}$ respectively, for all $p\in V$. Moreover, there is an undirected edge $(p,q)$ between any two neighboring pixels with weight $a_{p,q}$.
The weights $a_{s,p}$ and $a_{p,t}$ represent individual penalties for assigning pixel $p$ to the object and to the background respectively. On the other hand, given two pixels $p$ and $q$, the weight $a_{p,q}$ can be interpreted as a penalty for a discontinuity between their intensities.

In order to quantify the weights defined above, let us denote by $I_p\in [0,1]$ the intensity of pixel $p$. Then, see, e.g.,~\cite{boykov2006graph}, $a_{p,q}$ is computed as
\begin{align*}
a_{p,q} = e^{-\frac{(I_p-I_q)^2}{2\sigma^2}},
\end{align*}
where $\sigma$ is a constant modeling, e.g., the variance of the camera noise.
Moreover, weights $a_{s,p}$ and $a_{p,t}$ are respectively computed as
\begin{align*}
a_{s,p} =& -\lambda \log\Pb(x_p = 1)\\
a_{p,t} =& -\lambda \log\Pb(x_p = 0),
\end{align*}
where $\lambda>0$ is a constant and $\Pb(x_p=1)$ (respectively $\Pb(x_p=0)$) denotes the probability of pixel $p$ to belong to the foreground (respectively background).

The goal of the $s$--$t$ minimum cut problem is to find a subset $X\subseteq V$ of pixels such that the sum of the weights of the edges from $X\cup\{s\}$ to $\{t\}\cup V\setminus X$ is minimized.

\subsection{Distributed set-up}
In the considered distributed set-up, $N=8$ agents are connected
according to a strongly-connected Erd\H{o}s-R\'{e}nyi random digraph and each of them has access only to a portion of the image (see Figure~\ref{fig:agent_start}). In this set-up, clearly, each agent can assign weights only to some edges in $E_{s-t}$ so that, it cannot segment the entire image on its own.

Let $V_i\subseteq V$ be the set of pixels seen by agent $i$. Each node $i$ assigns a local intensity $I^i_p$ to each pixel $p\in V_i$. Then, it computes its local weights as
 \begin{align*}
   a^i_{p,q} &=
  \begin{cases}
   e^{-\frac{(I^i_p-I^i_q)^2}{2\sigma^2}}, & \text{if } p,q\in V_i\\
   0, & \text{otherwise}
 \end{cases}\\
   a^i_{s,p} &=
  \begin{cases}
   -\lambda \log\Pb(x^i_p = 1), & \text{if } p\in V_i\\
   0, & \text{otherwise}
 \end{cases}\\
   a^i_{p,t} &=
  \begin{cases}
 -\lambda \log\Pb(x^i_p = 0), & \text{if } p\in V_i\\
   0, & \text{otherwise}
 \end{cases}
 \end{align*}

Given the above locally defined weights, each agent $i$ construct its private submodular function $F_i$ as
\begin{align}
F_i(X)=
  \sum_{\substack{p\in X \\ q\in V\setminus X}}\!\!  a^i_{p,q}
  +\!\!
  \sum_{q\in V\setminus X}\!\!  a^i_{s,q}
  +
  \sum_{p\in X}  a^i_{p,t}
  -\!\!
  \sum_{q\in V}\!\! a^i_{s,q}.\label{eq:submod_local}
\end{align}
Here, the first term takes into account the edges from $X$ to $V\setminus X$, the
second one those from $s$ to $V\setminus X$,
and the third one those from $X$ to $t$. The last term is a normalization term
guaranteeing $F_i(\emptyset)=0$.
Then, by plugging~\eqref{eq:submod_local} in problem~\eqref{pb:standard}, the optimization problem that the agents have to cooperatively solve in order to segment the image is
\begin{equation*}
\begin{aligned}
&\m_{X\subseteq V}\sum_{i=1}^{N}\left(
  \sum_{\substack{p\in X \\ q\in V\setminus X}}\!\!  a^i_{p,q}
  +\!\!
  \sum_{q\in V\setminus X}\!\!  a^i_{s,q}
  +
  \sum_{p\in X}  a^i_{p,t}
  -\!\!
  \sum_{q\in V}\!\! a^i_{s,q}\right)\!\!.
  \end{aligned}
\end{equation*}

\begin{figure}
    \begin{center}
     \includegraphics[width=.7\columnwidth]{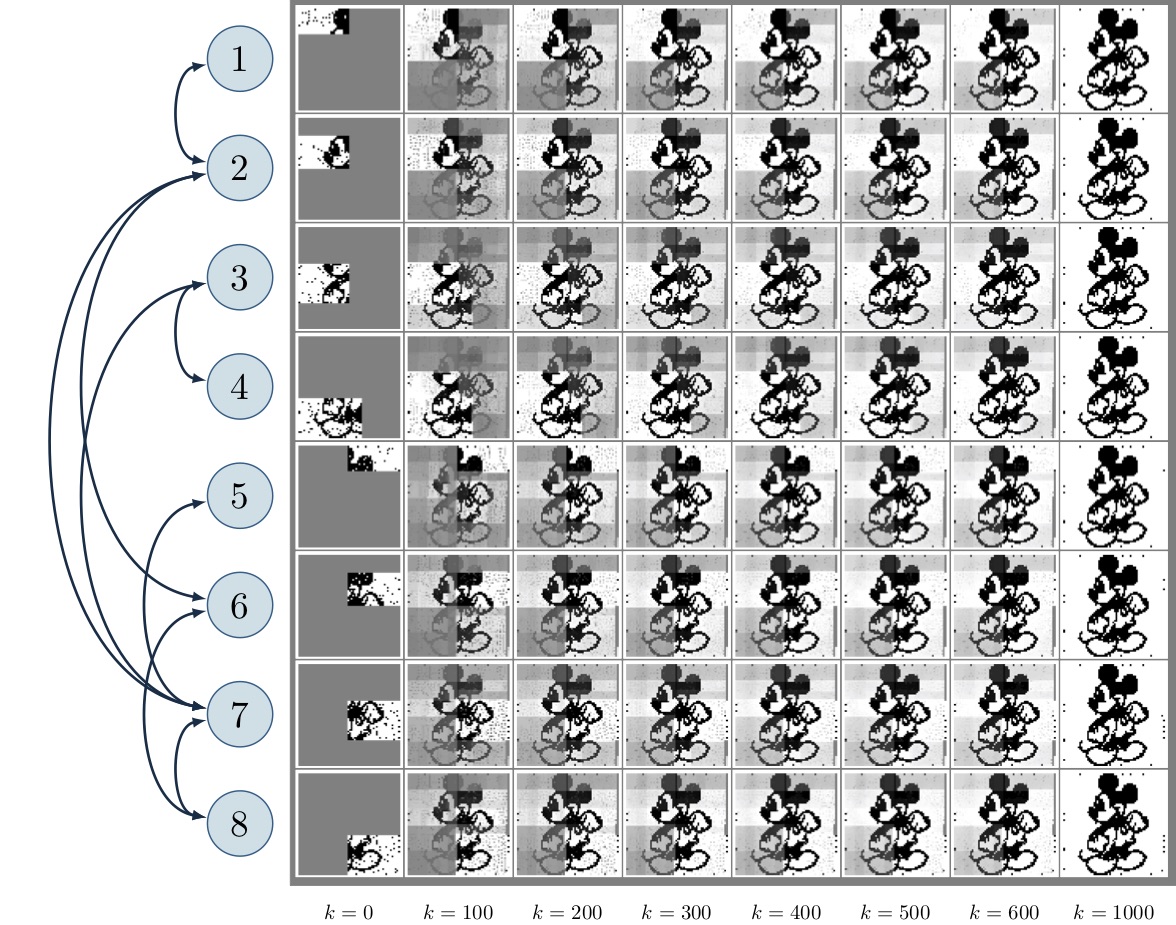}    %
    \caption{Cooperative image segmentation. Evolution of the local solution estimates for each agent in the network.}
    \label{fig:agent_stamps}
    \end{center}
\end{figure}

We applied the MICKY distributed algorithm to this set-up and we split the optimization variable in $40$ blocks.
In order to mimic possible errors in the construction  of the local weights,
we added some random noise to the image.
We implemented the MICKY algorithm by using the Python package DISROPT~\cite{farina2019disropt} and we ran it for $K=1000$ iterations. A graphical representation of the results is reported in Figure~\ref{fig:agent_stamps}.
 Each row is associated to one network agent while each column is associated to a different time stamp. More in detail, we show the initial condition at time $k=0$ and the candidate (continuous) solution at $k\in\{100,200,300,400,500,600\}$ iterations. The last column represents the solution $X_i^{\text{end}}$ of each agent obtained by thresholding $x_i^{k}$ with $k=1000$ and $\tau=0.5$.
As appearing in Figure~\ref{fig:agent_stamps}, the local solution set estimates $X_i^{\text{end}}$ are almost identical. Moreover, the connectivity structure of the network clearly affects the evolution of the local estimates.
Finally, the evolution of the cost error is depicted in Figure~\ref{fig:costconv}, where $X_i^k\triangleq \{i\mid x_i^k >\tau\}$.

 \begin{figure}[h!]
     \centering
     \includegraphics[scale=.78]{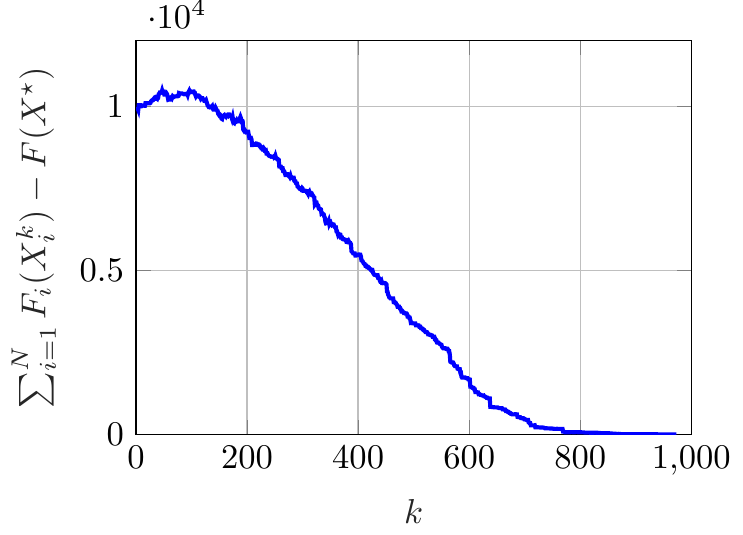}
     \caption{Numerical example. Evolution of the error between the cost computed at the (thresholded) local solution estimates and the optimal cost.}
     \label{fig:costconv}
 \end{figure}

\section{Conclusions}\label{sec:conclusions}
In this paper we presented MICKY, a distributed algorithm for solving submodular problems involving the minimization of the sum of many submodular functions without any central unit. It involves random block updates and communications, thus requiring a reduced local computational load and allowing its deployment on networks with low communication bandwidth (since it requires
a small amount of information to be transmitted at each iteration).
Its convergence in expected value has been shown under mild assumptions. The MICKY algorithm has ben tested on a cooperative image segmentation problem in which each agent has access to only a portion of the entire image.

%

\end{document}